\documentclass[12 pt]{amsart}
\usepackage{amscd,amsmath,amsthm,amssymb}
\usepackage{verbatim}
\usepackage[all]{xy}
\usepackage{color}
\usepackage{hyperref}
\usepackage{multirow}
\usepackage{subcaption}

\usepackage{tikz, float} \usetikzlibrary {positioning,calc,intersections}

\usepackage[lined,algonl,boxed,norelsize]{algorithm2e}
\let\chapter\undefined

\def\NZQ{\mathbb}               
\def\NN{{\NZQ N}}

\def\FF{{\NZQ F}}

\def\TT{{\NZQ T}}

\def\frk{\mathfrak}               

\def\Phi{{\frk n}}
\def\Phi{{\frk N}}

\def\kb{{\mathbf k}}

\def\xb{{\mathbf x}}
\def\yb{{\mathbf y}}

\def\pb{{\mathbf p}}

\def\A{{\mathcal A}}

\def\Fc{{\mathcal F}}

\def\Sc{{\mathcal S}}
\def\P{{\mathcal P}}

\def\C{{\mathcal C}}

\def\Dc{{\mathcal D}}
\def\DP{{\mathcal DP}}

\def\Fc{{\mathcal F}}

\def\Sc{{\mathcal S}}

\def\xb{{\mathbf x}}
\def\yb{{\mathbf y}}

\def\opn#1#2{\def#1{\operatorname{#2}}} 
\opn\chara{char} \opn\length{\ell} \opn\pd{projdim} \opn\rk{rk}
\opn\projdim{proj\,dim} \opn\injdim{inj\,dim} \opn\rank{rank}
\opn\depth{depth} \opn\grade{grade} \opn\height{height}
\opn\embdim{emb\,dim} \opn\codim{codim}

\opn\Tr{Tr} \opn\bigrank{big\,rank}
\opn\superheight{superheight}\opn\lcm{lcm}
\opn\trdeg{tr\,deg}
\opn\reg{reg} \opn\lreg{lreg} \opn\ini{in} \opn\lpd{lpd}
\opn\size{size} \opn\sdepth{sdepth}
\opn\link{link}\opn\fdepth{fdepth}\opn\lex{lex}
\opn\LM{LM}
\opn\LC{LC}
\opn\NF{NF}
\opn\Merge{Merge}
\opn\sgn{sgn}
\opn\suppPos{suppPos}
\opn\div{div} \opn\Div{Div} \opn\cl{cl} \opn\Pic{Pic}
\opn\Prin{Prin}
\opn\op{op}
\opn\indeg{indeg} \opn\outdeg{outdeg}
\opn\red{red}
\opn\Spec{Spec} \opn\Supp{Supp} \opn\supp{supp} \opn\Sing{Sing}
\opn\Ass{Ass} \opn\Min{Min}\opn\Mon{Mon}
\opn\Ann{Ann} \opn\Rad{Rad} \opn\Soc{Soc}
 \opn\Ker{Ker} \opn\Coker{Coker} \opn\Am{Am}
\opn\Hom{Hom} \opn\Tor{Tor} \opn\Ext{Ext} \opn\End{End}
\opn\Aut{Aut} \opn\id{id}

\opn\nat{nat}
\opn\pff{pf}
\opn\Pf{Pf} \opn\GL{GL} \opn\SL{SL} \opn\mod{mod} \opn\ord{ord}
\opn\Gin{Gin} \opn\Hilb{Hilb}\opn\sort{sort}
\opn\spn{span}
\opn\Image{Image}
\opn\aff{aff} \opn\con{conv} \opn\relint{relint} \opn\st{st}
\opn\lk{lk} \opn\cn{cn} \opn\core{core} \opn\vol{vol}
\opn\link{link} \opn\star{star}\opn\lex{lex}\opn\set{set}
\opn\dist{dist}
\opn\gr{gr}

\def\pot#1#2{#1[\kern-0.28ex[#2]\kern-0.28ex]}

\opn\dirlim{\underrightarrow{\lim}}
\opn\inivlim{\underleftarrow{\lim}}
\let\to=\rightarrow

\def\Implies{\ifmmode\Longrightarrow \else
        \unskip${}\Longrightarrow{}$\ignorespaces\fi}
\def\implies{\ifmmode\Rightarrow \else
        \unskip${}\Rightarrow{}$\ignorespaces\fi}
\def\iff{\ifmmode\Longleftrightarrow \else
        \unskip${}\Longleftrightarrow{}$\ignorespaces\fi}

\let\:=\colon
\newtheorem{Theorem}{Theorem}[section]
\newtheorem{Lemma}[Theorem]{Lemma}
\newtheorem{Corollary}[Theorem]{Corollary}
\newtheorem{Proposition}[Theorem]{Proposition}

\theoremstyle{remark}
\newtheorem{Remark}[Theorem]{Remark}

\theoremstyle{definition}
\newtheorem{Example}[Theorem]{Example}

\newtheorem{Definition}[Theorem]{Definition}

\let\kappa=\varkappa
\def\qed{\ifhmode\textqed\fi
      \ifmmode\ifinner\quad\qedsymbol\else\dispqed\fi\fi}
\def\textqed{\unskip\nobreak\penalty50
       \hskip2em\hbox{}\nobreak\hfil\qedsymbol
       \parfillskip=0pt \finalhyphendemerits=0}
\def\dispqed{\rlap{\qquad\qedsymbol}}

\opn\dis{dis}
\def\pnt{{\raise0.5mm\hbox{\large\bf.}}}

\opn\Lex{Lex}
\opn\syz{{\rm syz}}
\opn\spoly{{\rm spoly}}
\opn\LM{{\rm LM}}
\opn\lm{{\rm lm}}
\opn\projdim{{\rm projdim}}
\opn\lcm{{\rm lcm}} \opn\A{\mathcal A}

\opn\prob{{\rm prob}}

\numberwithin{equation}{section}

\tikzstyle{Cwhite}=[scale = .6,circle, fill = white, minimum size=2.5mm]
\tikzstyle{Cgray}=[scale = .4,circle, fill = gray, minimum size=3mm]
\tikzstyle{Cblack2}=[scale = .4,circle, fill = black, minimum size=3mm]
\tikzstyle{Cblack}=[scale = .7,circle, fill = black, minimum size=3mm]
\tikzstyle{C0}=[scale = .9,circle, fill = black!0, inner sep = 0pt, minimum size=3mm]
\tikzstyle{C1}=[scale = .7,circle, fill = black!0, inner sep = 0pt, minimum size=3mm]
\tikzstyle{Cred}=[scale = .4,circle, fill = red, minimum size=3mm]
\tikzstyle{Cblue}=[scale = .4,circle, fill =blue, minimum size=3mm]

\begin{document}

\title{Polarization and depolarization of monomial ideals with application to multi-state system reliability}
\author{Fatemeh Mohammadi} 
\address{School of Mathematics, University of Bristol, UK}
\email{fatemeh.mohammadi@bristol.ac.uk}

\author{Patricia Pascual-Ortigosa} 
\address{Departamento de Matem\'aticas y Computaci\'on,  Universidad de La Rioja, Spain}
\email{papasco@unirioja.es}

\author{Eduardo S\'aenz-de-Cabez\'on}
\address{Departamento de Matem\'aticas y Computaci\'on,  Universidad de La Rioja, Spain}
\email{eduardo.saenz-de-cabezon@unirioja.es}

\author{Henry P. Wynn}
\address{Department of Statistics, London School of Economics, UK}
\email{h.wynn@lse.ac.uk}


\maketitle
\begin{abstract}
Polarization is a powerful technique in algebra which provides combinatorial tools to study algebraic invariants of monomial ideals. We study the reverse of this process, depolarization which leads to a family of ideals which share many common features with the original ideal. Given a squarefree monomial ideal, we describe a combinatorial method to obtain all its depolarizations, and we highlight their similar properties such as graded Betti numbers.  We show that even though they have many similar properties, their differences in dimension make them distinguishable in applications in system reliability theory. In particular, we apply polarization and depolarization tools to study the reliability of multi-state coherent systems via binary systems and vice versa. We use depolarization as a tool to reduce the dimension and the number of variables in coherent systems.

\end{abstract}
\section{Introduction}\label{sec:intro}
\subsection{Background.} Polarization is an operation that transforms a monomial ideal into a squarefree monomial ideal in a larger polynomial ring, preserving several important features of the original ideal such as graded Betti numbers. The main idea behind polarization is the possibility of using the combinatorial properties of squarefree monomial ideals when studying problems about general monomial ideals. Polarization is used in a wide variety of applications in the theory of monomial ideals. An important feature of polarization is that it was used by Hartshorne to prove the connectedness of the Hilbert scheme by showing that distractions of ideals can be described as specializations of polarizations of monomial ideals \cite{H66}. One of the main applications is its use to study the Cohen-Macaulay property of monomial ideals by passing to squarefree monomial ideals and applying Reisner's criterion on their associated simplicial complex \cite{M08,S83}. It is also used to study associated primes of monomial ideals and their powers \cite{HHQ15,HM10,MMV12}. 
\vspace{-.27cm}
\subsection{Our contribution.}  Even though polarization has been used as a powerful tool in algebraic geometry and in applications, the inverse operation, depolarization, has been less investigated.  Depolarization can be used to study the algebraic invariants of squarefree monomial ideals using general monomial ideals in less variables \cite{P06}. We note that depolarization is not unique, in the sense that a given squarefree monomial ideal might have different depolarizations. 

A main goal of this paper is to find all depolarizations of a given squarefree monomial ideal and describe their structure combinatorially. This is achieved in Section \ref{sec:depolarization}, which contains our main theoretical results. More precisely, for any squarefree monomial ideal $I$, we define the so-called depolarization orders, and we show that any such order gives rise to a depolarization of $I$, see Proposition \ref{prop:oneDepolarization}. Moreover, we show that any depolarization of $I$ can be constructed this way, see Theorem \ref{th:allDepolarizations}. As a given squarefree monomial ideal $I$ shares several important features with its depolarizations, the aforementioned combinatorial characterization can be used to select a convenient depolarization of $I$ in order to study the properties of either $I$ or any of its depolarizations. For example, one immediately obtains the Hilbert function of these ideals by studying only one of them, as they are closely related. It is also interesting to study the behaviour of other properties and features that are not shared within the family of depolarizations of $I$. We will use such dissimilarities to identify a particular depolarization whose invariants are easier to compute and provide information about all depolarizations of $I$. See, e.g., Proposition~\ref{prop:npaths}.
\vspace{-.27cm}

\subsection{Applications in system reliability theory} In applications it is sometimes convenient to work with squarefree monomial ideals, i.e, with polarizations of monomial ideals, and use all their features as combinatorial objects, as seen in \cite{M08, M2016}. However, on many occasions it makes sense to work on depolarizations of squarefree monomial ideals and reduce the number of variables of their corresponding rings. See, e.g.,  \cite{BM09, MS16} for similar considerations in different contexts. We propose to explore both directions in the context of algebraic analysis of the reliability of systems. In previous works \cite{GW04,MSW17,SW09,SW10,SW11,SW12,SW15,M16} the authors have studied the ideals associated to coherent systems and used their algebraic invariants such as Hilbert function and Betti numbers to compute the reliability of such systems. However, most of the work is devoted to binary systems whose associated ideals are squarefree. But, in practice many systems are non-binary, i.e., their components have multiple possible states, and hence their associated ideals are 
not squarefree. In this paper, we use polarization tools to study the reliability of a general multi-state system via binary systems, and we use depolarization as a dimension and variable reduction to study the reliability of  binary systems.
\vspace{-.25cm}

\subsection{Structure of the paper.} Section \ref{sec:preliminaries} gives the necessary preliminaries on polarization and depolarization. In Section \ref{sec:depolarization} we introduce the support posets as combinatorial tools to explore all depolarizations of a given squarefree monomial ideal. In Theorem~\ref{th:allDepolarizations} we show that every depolarization of a given monomial ideal $I$ can be obtained from its support poset. We describe the structure of depolarizations of $I$ in terms of its associated poset. This gives rise to a new bound for the projective dimension of monomial ideals, see Theorem \ref{th:pdBound}. Moreover, we describe several families of ideals for which there exists at least one quasi-stable ideal among their depolarizations. The algebraic invariants of quasi-stable ideals are easier to compute, see, e.g., \cite{S09b}. Therefore, we can compute the algebraic invariants of such ideals using their corresponding quasi-stable ideals.
In Section \ref{sec:systems} we turn to algebraic studies of reliability of networks and we describe how one can apply polarization and depolarization tools to compute the reliability of coherent multi-state systems. Finally, we give several examples applying the depolarization tools developed through this paper for dimensional reduction in coherent systems.

\section{Polarization and depolarization}\label{sec:preliminaries}
Throughout this paper, we will assume that $R=\kb[x_1,\dots,x_n]$ is a polynomial ring in $n$ indeterminates over a field $\kb$ on which we make no explicit assumptions. For any monomial ideal $I\subseteq R$, we let $G(I)=\{ m_1,\dots,m_r\}$ be the unique minimal monomial generating set of $I$.

\begin{Definition}\label{def:polarization}
Let $a=(a_1,\dots,a_n)$ and $\mu=(b_1,\dots,b_n)$ be two elements in $\NN^n$ with $b_i\leq a_i$ for all $i$. The polarization of $\mu$ in $\NN^{a_1+\cdots+a_n}$ is the multi-index $$\overline{\mu}=(\underbrace{1,\dots,1}_{b_1},\underbrace{0,\dots,0}_{a_1-b_1},\dots,\underbrace{1,\dots,1}_{b_n},\underbrace{0,\dots,0}_{a_n-b_n}).$$ The {\em polarization of} $\xb^\mu=x_1^{b_1}\cdots x_n^{b_n}\in R$ with respect to $a$ is the squarefree monomial $\xb^{\overline{\mu}}=x_{1,1}\cdots x_{1,b_1}\cdots x_{n,1}\cdots x_{n,b_n}$ in $S=\kb[x_{1,1},\dots,x_{1,a_1},\dots,x_{n,1},\dots,x_{n,a_n}]$. Note that for ease of notation we used $\xb$ with two different meanings in this definition.
Let $I=\langle m_1,\dots,m_r\rangle\subseteq R$ be a monomial ideal and let $a_i$ be the maximum exponent to which indeterminate $x_i$ appears among the generators of $I$.
The {\em polarization of $I$}, denoted by $I^P$, is the monomial ideal in $S$ given by $I^P=\langle \overline{m_1},\dots,\overline{m_r}\rangle$, where $\overline{m_i}$ is the polarization of $m_i$ with respect to $a$. 
\end{Definition}

Note that Definition \ref{def:polarization} is a combinatorial expression of the following result of Fr\"oberg \cite{F82} as given in \cite{V01} in which $I'$ is a polarization of $I$.

\begin{Proposition}\label{prop:Froberg}
For any monomial ideal $I\subset R$ there is a squarefree monomial ideal $I'\subset R'$ such that $R/I=R'/(I'+(\underline{h}))$, where $\underline{h}$ is a regular sequence on $R'/I'$ of forms of degree one.
\end{Proposition}
Before defining the depolarization of ideals, we would like to note that as stated in Section~1, the polarization of monomial ideals have been extensively studied in the commutative algebra literature. Here we provide a thorough study of the reverse process, called depolarization,  to provide algebraic tools to compute the reliability of large networks by reducing the number of variables and the dimension of networks.
\begin{Definition}
Let $R, S$ and $T$ be polynomial rings over the field $\kb$. Let $I\subseteq R$ be a squarefree monomial ideal. A {\em depolarization of $I$} is a monomial ideal $J\subseteq S$ such that $I$ is isomorphic to $J^P\subseteq T$ that is:
There is a bijective map $\varphi$ from the set of variables of $R$ to the set of variables of $T$ such that $\varphi(G(I))=G(J^P)$, 
 where $G(J^P)$ is the unique minimal monomial generating set of $J^P$.
\end{Definition}
Note that the rings $R$ and $T$ above should have the same number of variables.

\begin{Example}\label{ex:depolarization}
Consider the squarefree monomial ideal $I=\langle xyz,xyt,yzt,ytu\rangle\subseteq R= \kb[x,y,z,t,u]$. The ideals $J=\langle ab^2,a^2b,abc,a^2c \rangle$ and $J'=\langle ab^2,abc,b^3,b^2c\rangle$ in $S=\kb[a,b,c]$ are two different depolarizations of $I$.

To check this observe that $J^P=\langle a_1b_1b_2,a_1a_2b_1,a_1b_1c_1,a_1a_2c_1\rangle\subseteq \kb[a_1,a_2,b_1,b_2,c_1]$ and we have an isomorphism between $I$ and $J^P$ via the correspondence $a_1\mapsto y$, $a_2\mapsto x$, $b_1\mapsto t$, $b_2\mapsto u$, $c_1\mapsto z$. On the other hand, $J'^P=\langle a_1b_1b_2,a_1b_1c_1,b_1b_2b_3,b_1b_2c_1\rangle\subseteq \kb[a_1,b_1,b_2,b_3,c_1]$ is isomorphic to $I$ by $a_1\mapsto x$, $b_1\mapsto y$, $b_2\mapsto t$, $b_3\mapsto u$, $c_1\mapsto z$.
\end{Example}

\begin{Remark}
As seen in Proposition~\ref{prop:Froberg}, depolarization is a combinatorial way to perform identification of variables arisen from a regular sequence of linear forms.  A natural question would be whether every such identification of variables can be read as a depolarization of the original ideal. In the following example, we show that this is not true in general.  Consider the following three ideals from \cite[Example~9.5]{MS16}: 
\begin{align*}
M&=\langle x_1^3,x_2^2,x_3^2,x_1^2x_2,x_2^2x_3,x_1x_2x_3\rangle\subseteq \kb[x_1,x_2,x_3] \\
\mathcal{M}&=\langle x_{12}x_{13}x_{14},x_{21}x_{24},x_{31}x_{34},x_{13}x_{14}x_{24},x_{12}x_{14}x_{34},x_{14}x_{24}x_{34}\rangle \\
&\subseteq \kb[x_{12},x_{13},x_{14},x_{21},x_{24},x_{31},x_{34}]\\
\mathcal{O}&=\langle x_{12}x_{13}x_{14},x_{12}x_{24},x_{13}x_{34},x_{13}x_{14}x_{24},x_{12}x_{14}x_{34},x_{14}x_{24}x_{34}\rangle \\
& \subseteq \kb[x_{12},x_{13},x_{14},x_{24},x_{34}].
\end{align*}
Both ideals $M$ and $\mathcal{O}$ can be obtained from $\mathcal{M}$ by identifying a set of variables together. More precisely, in $M$ we relabel every variable $x_{ij}$ with the variable $x_i$, and in $\mathcal{O}$ we identify the following sets of variables with each other $\{x_{13},x_{31}\}$ and $\{x_{12},x_{21}\}$.  Lemma~10.4 from \cite{MS16} implies that these identifications of variables arise from a regular sequence of linear forms. However, we note that neither of them is a depolarization of $\mathcal{M}$.
\end{Remark}

\section{Combinatorics of depolarization ideals}\label{sec:depolarization}
\subsection{The support poset}
Let $R=\kb[x_1,\dots,x_n]$ be a polynomial ring in $n$ variables. For any monomial $m$ of $R$ the {\em support} of $m$, denoted by $\supp(m)$, is defined as the set of indices of variables which divide $m$. The support of a monomial ideal $I\subseteq R$ is $\supp(I)=\bigcup_{m\in G(I)}\supp(m)$, where $G(I)$ is the unique minimal monomial generating set of $I$. We say that an ideal $I$ has {\em full support} if $\supp(I)=\{1,\dots,n\}=[n]$.  For ease of notation we assume that ideals have full support, unless otherwise stated.

Let $I$ be a squarefree monomial ideal with $G(I)=\{m_1,\dots,m_r\}$. 
For each $i$ in $\supp(I)$ we define the set $C_i\subseteq \supp(I)$ as,
\[
C_i=\{  j\,\vert \, j\in \bigcap_{m\in G(I)}\{\supp(m)\vert x_i \mbox{ divides } m\} \}.
\]
In other words, $C_i$ is given by the indices of all the variables that appear in every minimal generator of $I$ in which $x_i$ is present.
Let $C_I=\{C_1,\ldots,C_n\}$. The poset on the elements of $C_I$ ordered by inclusion is called the {\em support poset} of $I$ and is denoted $\suppPos(I)$.
We define the {\em support poset} of a general monomial ideal as the support poset of its polarization obtained from Definition~\ref{def:polarization}.

Given $n$ subsets $C_i$ of $[n]$, 
we form the poset $(\C=\{C_1,\dots,C_n\},\subseteq)$ on elements $C_i$ which are ordered by inclusion. Note that some $C_i$ can possibly be equal to $C_j$ for $i\neq j$.
A natural question is whether for such $(\C,\subseteq)$ we can construct a monomial ideal $I_\C$ whose support poset is $(\C,\subseteq)$.  This question is not easy in general. See Example~\ref{ex:support-posets} (2) for a counterexample. In the following proposition we provide a sufficient condition to construct such ideals.
Another sufficient condition will be given in Proposition~\ref{prop:npaths}.
\begin{Proposition}\label{prop:supportIdeals}
Let $(\C=\{C_1,\dots,C_n\},\subseteq)$ be a poset such that 
$\{i\}\subseteq C_i\subseteq [n]$ for each $i$, and if $k\in C_i$ and $i\in C_j$ then $k\in C_j$ for all $i,j,k$. 
Let $R=\kb[x_1,\dots,x_n]$ and let $m_i=\prod_{j\in C_i}x_j$ for each $i$. 
For any $\sigma\subseteq [n]$ let $m_\sigma=\lcm(m_i \vert i\in\sigma)$, and for any collection $\Sigma$ of subsets of $[n]$, consider the monomial ideal $I_\Sigma=\langle m_\sigma \vert \sigma\in\Sigma\rangle$.
Then $(\C,\subseteq)$ is the support poset of $I_\Sigma$ if the following properties hold:
\begin{enumerate}
\item $\forall i\in[n]$ there is some $\sigma \in\Sigma$ such that $x_i\vert m_\sigma$.
\item If $\{\sigma:\ x_i|m_\sigma\}\subseteq \{\sigma:\ x_j|m_\sigma\}$, then $C_j\subseteq C_i$.
\end{enumerate}
\end{Proposition}

\begin{proof}
Let $(\Dc=\lbrace D_1,\dots, D_n\rbrace, \subseteq)$ be the support poset of the ideal $I_\Sigma$. We want to show that $D_j=C_j$, for all $j\in\lbrace 1,\dots,n\rbrace$. Note that 
$
D_j=\{ k\, \vert \, k \in \bigcap_{x_j\vert m_\sigma\atop\sigma\in\Sigma}\supp(m_\sigma)\}.
$
Let $k\in C_j$. For any $\sigma$ with $x_j|m_\sigma$, there is some $\ell\in\sigma$ such that $x_j\vert m_\ell$. Hence, $j\in C_\ell$ which implies that $k\in C_\ell$ and so $x_k\vert m_\sigma$ and $k\in D_j$.

On the other hand, $k\in D_j$ implies that $x_k$ divides all $m_\sigma$, where $x_j\vert m_\sigma$. This together with condition $(2)$ imply that $C_k\subseteq C_j$ and $k\in C_j$ which means $D_j\subseteq C_j$.
%
%
%
%
\end{proof}

In the following, we show that some posets might not appear as support poset of any ideal, and on the other hand, several ideals might have the same support poset.

\begin{Example}\label{ex:support-posets}
\begin{enumerate}
\item Let 
$C_1=\{1,2\},\,C_2=\{2\},\, C_3=\{3\}, C_4=\{4\}$ and $C_5=\{4,5\}$. Let $\Sigma_1=\{\{1\},\{2,4\},\{3\},\{5\}\}$, $\Sigma_2=\{\{1\},\{2,3\},\{3,4\},\{5\}\}$ and $\Sigma_3=\{\{1,3\},\{3,5\},\{1,4\},\{2,5\}\}$. These three collections satisfy the conditions in Proposition \ref{prop:supportIdeals} and hence $(\C=\{C_1,\ldots,C_5\},\subseteq)$ is the support poset of the ideals $I_{\Sigma_1}=\langle x_1x_2,x_2x_4,x_3,x_4x_5\rangle$, $I_{\Sigma_2}=\langle x_1x_2,x_2x_3,x_3x_4,x_4x_5\rangle$ and $I_{\Sigma_3}=\langle x_1x_2x_3,x_3x_4x_5,x_1x_2x_4,x_2x_4x_5\rangle$. 
\item Let $\C$ be given by $C_1=\{1\},\, C_2=\{1,2\}$ and $C_3=\{1,2,3\}$, then there is no monomial ideal $I\subseteq R[x_1,x_2,x_3]$ such that $(\C,\subseteq)$ is the support poset of $I$. To see this, observe that $x_1x_2x_3$ must be one of the minimal generators of $I$, hence the only one, but $\C$ is not the support poset of $I=\langle x_1x_2x_3\rangle$.
\item Let $C_1=\{ 1,2,4\}$, $C_2=\{ 1,2,4\}$, $C_3=\{ 1,2,3,4\}$, $C_4=\{ 4\}$, $C_5=\{ 1,2,4,5,6\}$, $C_6=\{ 4,6\}$, $C_7=\{7\}$, $C_8=\{7,8\}$, $C_9=\{ 7,8,9\}$, $C_{10}=\{ 7,8,10\}$. Then for $\Sigma=\{\{3\},\{6,7\},\{5\},\{9\},\{10\}\}$, the ideal 
\[
I_\Sigma=\langle x_1x_2x_3x_4,x_4x_6x_7,x_1x_2x_4x_5x_6,x_7x_8x_9,x_7x_8x_{10}\rangle\subseteq \kb[x_1,\dots,x_{10}].
\] 
has $(\mathcal{C}=\{C_1,\dots,C_{10}\},\subseteq)$ as its support poset. 
\end{enumerate}
\end{Example}

\begin{Remark}
Note that in any support poset, $k\in C_i$ and $i\in C_j$ imply that $k \in C_j$. We can use this fact to visualize support posets using their Hasse diagrams, where each node is labelled by their elements which are not in any of the nodes below it.

\begin{figure}[h]
\centering
\begin{subfigure}[b]{0.48\textwidth}
\centering
      \begin{tikzpicture}[scale=.6, transform shape, vertices/.style={text width= 1.5em,align=center,draw=black, fill=white, circle, inner sep=1pt}]
             \node [vertices] (0) at (0,0){2};
             \node [vertices] (2) at (1.5,0){3};
             \node [vertices] (3) at (3,0){4};
             \node [vertices] (1) at (0,1.5){1};
             \node [vertices] (4) at (3,1.5){5};
              \foreach \to/\from in {0/1, 3/4}
     \draw [-] (\to)--(\from);
     \end{tikzpicture}
     \caption{Support poset for Example \ref{ex:support-posets} (1)}
     \label{fig:Hasse-support-a}
 \end{subfigure} 
\begin{subfigure}[b]{0.48\textwidth}
\centering
 \begin{tikzpicture}[scale=.6, transform shape, vertices/.style={text width= 1.5em,align=center,draw=black, fill=white, circle, inner sep=1pt}]
             \node [vertices] (0) at (3,0){4};
             \node [vertices] (1) at (7.5,0){7};
             \node [vertices] (2) at (1.5,1.5){1,2};
             \node [vertices] (3) at (4.5,1.5){6};
             \node [vertices] (4) at (7.5,1.5){8};
             \node [vertices] (5) at (0,3){3};
             \node [vertices] (6) at (3,3){5};
             \node [vertices] (7) at (6,3){9};
             \node [vertices] (8) at (9,3){10};
              \foreach \to/\from in {0/2, 0/3, 1/4, 2/5, 2/6, 3/6, 4/7,4/8}
     \draw [-] (\to)--(\from);		
     \end{tikzpicture}
     \caption{Support poset for Example \ref{ex:support-posets} (3)}
 \end{subfigure}
 \end{figure}
 \end{Remark}

The support poset of any monomial ideal $I\subseteq R=\kb[x_1,\dots,x_n]$, together with a given ordering $<$ on the variables $x_1,\dots, x_n$ induces a partial order $\prec$ on the set of variables as follows: $x_i\prec x_j$ if $C_i\subset C_j$ or if $C_i=C_j$ and $x_i < x_j$. We call this partial order the {\em $<$-support poset} of $I$ and denote it by $\suppPos_<(I)$. 
If $C_i\neq C_j$ for every pair of indices, then $\suppPos(I)$ is equal to the $<$-support poset of $I$ for any order $<$. 
See Figure~\ref{fig:Hasse-support-a} for Example \ref{ex:support-posets} (1). 

Note that, the Hasse diagram of $\suppPos_<(I)$ can be obtained from the Hasse diagram of $\suppPos(I)$ in which every node $C$ labelled with more than one index is substituted by a vertical line of nodes labelled by distinct elements of $C$, ordered by $<$. In other words, every $<$-support poset of $I$ is a refinement of $\suppPos(I)$. See Figure~\ref{fig:suppPosExa} as a $<$-support poset of $I_\Sigma$ in Example \ref{ex:support-posets} (3) for any order on the variables which is compatible with $x_1<x_2$.

\begin{figure}[h]
\centering
 \begin{tikzpicture}[scale=.6, transform shape, vertices/.style={text width= 1.5em,align=center,draw=black, fill=white, circle, inner sep=1pt}]
             \node [vertices] (0) at (3,0){4};
             \node [vertices] (1) at (8.5,0){7};
             \node [vertices] (2) at (1.5,1.5){1};
             \node [vertices] (2b) at (1.5,3){2};
             \node [vertices] (3) at (4.5,2.25){6};
             \node [vertices] (4) at (8.5,1.5){8};
             \node [vertices] (5) at (0,4.5){3};
             \node [vertices] (6) at (3,4.5){5};
             \node [vertices] (7) at (7,3){9};
             \node [vertices] (8) at (10,3){10};
              \foreach \to/\from in {0/2, 0/3, 1/4, 2/2b, 2b/5, 2b/6, 3/6, 4/7,4/8}
     \draw [-] (\to)--(\from);		
     \end{tikzpicture}
     \caption{$\suppPos_<(I)$ for Example \ref{ex:support-posets} (3) for any order with $x_1<x_2$}\label{fig:suppPosExa}
 \end{figure}
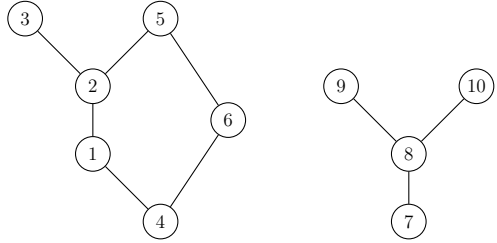
\subsection{Depolarization orders}

Recall that a subset $C$ of a poset $(\P,\prec)$ is a {\em chain} if any two elements of $C$ are comparable. We say that a chain $C$ of $(\P,\prec)$ is a {\em path} if there is no element $p\notin C$ such that $\min(C)\prec p \prec \max(C)$ and $p$ is comparable to every element in $C$. In other words, a path is a chain with no gaps, i.e., an interval within a chain. An antichain is a set of pairwise incomparable elements in $(\P,\prec)$.

\begin{Definition}
Given an order $<$ on the variables of $R$, a {\em depolarization order} of a squarefree monomial ideal $I\subseteq R$ is a partition of 
$\suppPos_<(I)$ into disjoint paths.
\end{Definition}

We now show that depolarization orders characterize all depolarizations of $I$. Namely, every depolarization order gives rise to a depolarization of $I$, and every depolarization of $I$ can be realized as a depolarization obtained by such an order.

\begin{Proposition}\label{prop:oneDepolarization}
Using any depolarization order of a squarefree monomial ideal $I$, we can construct a depolarization of $I$. 
\end{Proposition}
\begin{proof}
Let $({\P},\prec)$ be a depolarization order for a squarefree monomial ideal $I\subseteq R=\kb[x_1,\dots,x_n]$, where ${\P}=\{\sigma_1,\dots,\sigma_k\}$ and each $\sigma_i$ is a path in $\suppPos_<(I)$ for a given order $<$ on the variables of $R$. We construct a depolarization $J$ of $I$ in a polynomial ring $S=\kb[y_1,\dots,y_k]$ as follows: for each monomial $m$ in $G(I)$ consider the monomial $m'$ given by the image of $m$ under the correspondence $x_i\mapsto y_j$ for each $i\in\sigma_j$. The monomials $m'$ generate an ideal $J$ whose polarization $J^P$ is clearly equivalent to $I$ via the map sending $y_{j,\ell}\mapsto x_{{\sigma_j}_\ell}$ where ${\sigma_j}_\ell$ is the $\ell$-th element of $\sigma_j$ under the order $\prec$.
\end{proof}

\begin{Example}\label{ex:depolarizationOrder}
The partition given by $\P=\{\{4,2,1,3\}, \{6,5\}, \{7,8,9\}, \{10\}\}$ is a depolarization order for the ideal $I$ in Example \ref{ex:support-posets} (3) for any ordering in which $x_2< x_1$. Figure \ref{fig:partitionPoset} shows this partition.
\begin{figure}[h]
\centering
\begin{tikzpicture}[scale=.8, transform shape, vertices/.style={text width= 1.5em,align=center,draw=black, fill=white, circle, inner sep=1pt}]
             \node [vertices] (0) at (3,0){4};
             \node [vertices] (1) at (7.5,0){7};
             \node [vertices] (2) at (1.5,1.5){2};
             \node [vertices] (2b) at (1.5,3){1};
             \node [vertices] (3) at (4.5,2.25){6};
             \node [vertices] (4) at (7.5,1.5){8};
             \node [vertices] (5) at (0,4.5){3};
             \node [vertices] (6) at (3,4.5){5};
             \node [vertices] (7) at (6,3){9};
             \node [vertices] (8) at (9,3){10};
              \foreach \to/\from in {0/2, 0/3, 1/4, 2/2b, 2b/5, 2b/6, 3/6, 4/7,4/8}
     \draw [-] (\to)--(\from);
     
     \draw[name path=ellipse,red,thick, rotate around={-55:(1.5,2.5)}]
		(1.5,2.5) circle[x radius = 4 cm, y radius = 1.1 cm];
     \draw[name path=ellipse,red,thick, rotate around={-55:(3.75,3.25)}]
		(3.75,3.25) circle[x radius = 3 cm, y radius = .8 cm];
     \draw[name path=ellipse,red,thick, rotate around={-60:(7,1.5)}]
		(7,1.5) circle[x radius = 2.5 cm, y radius = 1.1 cm];
      \draw[name path=ellipse,red,thick]
		(9,3) circle[x radius = .8 cm, y radius = .8 cm];
     \end{tikzpicture}
\caption{A path partition of $\suppPos_<(I)$ in Example \ref{ex:support-posets} (3) gives a depolarization order $(\P,\prec)$ for $I$.}
\label{fig:partitionPoset}
\end{figure}
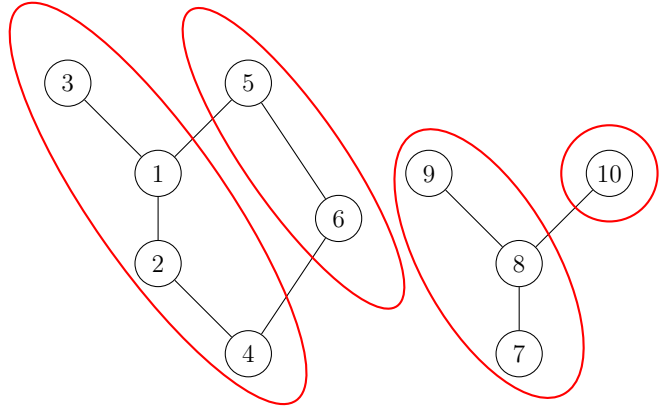

The depolarization order $(\P,\prec)$ depicted in Figure \ref{fig:partitionPoset} gives the depolarization $J=\langle y_1^4,y_1y_2y_3,y_1^3y_2^2,y_3^3,y_3^2y_4 \rangle\subseteq \kb[y_1,y_2,y_3,y_4]$ of $I$. The equivalence between $J^P$ and $I$ is given by $y_{1,1}\mapsto x_4$, $y_{1,2}\mapsto x_2$, $y_{1,3}\mapsto x_1$, $y_{1,4}\mapsto x_3$, $y_{2,1}\mapsto x_6$, $y_{2,2}\mapsto x_5$, $y_{3,1}\mapsto x_7$, $y_{3,2}\mapsto x_8$, $y_{3,3}\mapsto x_9$, $y_{4,1}\mapsto x_{10}$.
\end{Example}

We have just seen that every depolarization order of a squarefree monomial ideal $I$ gives a depolarization of $I$. Now, we study the reverse of this process and we show that given any depolarization $J$ of $I$ we can explicitly find a depolarization order from which we can reconstruct $J$.

\begin{Theorem}\label{th:allDepolarizations}
Let  $I=\langle m_1,\dots,m_r\rangle\subseteq R=\kb[x_1,\dots,x_n]$ be a squarefree monomial ideal. Every depolarization of $I$ can be obtained from a depolarization order of $I$.
\end{Theorem}

\begin{proof}
Let $J\subseteq S=\kb[y_1,\dots,y_k]$ be a depolarization of the ideal $I$ and let $J^P\subseteq T=\kb[y_{1,1},\dots,y_{1,j_1},\dots,y_{k,1},\dots,y_{k,j_k}]$ be the polarization of $J$. Since $J$ is a depolarization of $I$, we know that $R$ and $T$ have the same number of variables and that $I$ and $J^P$ are equivalent under a map sending $x_i$ to $y_{a,b}$ for some $a\in\{1,\dots,k\}$ and $b\in\{1,\dots,j_a\}$. Now consider in $\{1,\dots,n\}$ the partition $\P$ with $k$ subsets in which $\sigma_i$ contains all $j$ such that $x_j$ corresponds to some $y_{i,b}$ with the total order given by $j<j'$ if $b<b'$, where $ y_{i,b}\mapsto x_j$ and $y_{i,b'}\mapsto x_{j'}$. Then $(\P,<)$ is a depolarization order for $I$ that produces the depolarization $J$.
\end{proof}

\begin{Example}
Consider the depolarization $J=\langle ab^2,a^2b,abc,a^2c\rangle$ of the ideal $I=\langle xyz,xyt,yzt,ytu\rangle$ in Example \ref{ex:depolarization}. We have that $J^P\subseteq\kb[a_1,a_2,b_1,b_2,c_1]$ is equivalent to $I\subseteq\kb[x,y,z,t,u]$ through the correspondence $a_1\mapsto y$, $a_2\mapsto x$, $b_1\mapsto t$, $b_2\mapsto u$, $c_1\mapsto z$. The corresponding depolarization order is $\P=\{ \{y,x\},\{t,u\},\{z\}\}$ where the elements in the  sets are given in increasing order.
\end{Example}

\subsection{Depolarization posets}
Let $P$ and $P'$ be two path partitions of a given poset. We say that $P$ is a {\em refinement} of $P'$ if  for every path $C$ in $P$ there is a path $C'$ in $P'$ such that {$C'\subseteq C$}. The set of all path partitions of a given poset are sorted by refinement and using this ordering they form themselves a poset. Let $I$ be a squarefree monomial ideal and let $J$, $J'$ be two depolarizations of $I$. We say that $J\leq J'$ if the path partition giving rise to $J$ is a refinement of the one corresponding to $J'$. Using this ordering, a collection of ideals that are depolarizations of a given squarefree monomial ideal $I$ forms a poset in which $I$ is the unique minimal element. We call this the {\em depolarization poset} of $I$, denoted $\DP(I)$. Given any monomial ideal $J$ (not necessarily squarefree), we define its {\em depolarization poset} to be the depolarization poset of its polarization $J^P$. In other words, $\DP(J):=\DP(J^P)$.

Every depolarization poset has a unique minimal element which is a squarefree monomial ideal, hence $\DP(J)$ is a meet-semilattice for every monomial ideal $J$, that is for every pair $K$ and $K'$ in $\DP(J)$ there is an element in $\DP(J)$, denoted by $K\land K'$, which is smaller than both of them. On the other hand, $\DP(J)$ might have several maximal elements and therefore it is not a lattice in general. We say that an ideal $J\subseteq \kb[x_1,\dots,x_n]$ is a {\em maximum} element in its depolarization poset if there is no other ideal $J'\subseteq \kb[x_1,\dots,x_m]$ in $\DP(J)$ such that $m<n$. That means the ambient ring of $J$ has the minimal number of variables among the ambient rings of all ideals in $\DP(J)$.


\subsection{Copolar ideals} 
\begin{Definition}
Two monomial ideals $I$ and $J$ are called {\em copolar} if their polarizations are equivalent, i.e., they are in the same depolarization poset.
\end{Definition}
Copolarity is an equivalence relation in the set of monomial ideals. We say that a property of an ideal is {\em copolar} if it is shared by all ideals in the same polarity class. The following proposition gives a list of copolar properties.

\begin{Proposition}[Corollary 1.6.3 in \cite{HH10}]\label{prop:copolarProperties}
Let $I\subseteq S$ be a monomial ideal and let $J\subseteq T$ be its polarization. Then
\begin{enumerate}
\item $\beta_{i,j}(I)=\beta_{i,j}(J)$ for all $i$ and $j$
\item $H_{I}(t)=(1-t)^\delta H_{J}(t)$ where $\delta=\dim T-\dim S$
\item $\height(I)=\height(J)$
\item $\projdim(S/I)=\projdim(T/J)$ and $\reg(S/I)=\reg(T/J)$
\item $S/I$ is Cohen-Macaulay (resp. Gorenstein) if and only if $T/J$ is Cohen-Macaulay (resp. Gorenstein).
\end{enumerate}
\end{Proposition}

A reason behind the items in Proposition \ref{prop:copolarProperties} is that the {\em lcm-lattice} \cite{GPW99,MS04} of $I$, denoted by $\lcm(I)$, is isomorphic to the {\em lcm-lattice} of $J$ under the map taking $\lcm(m,m')$ to $\lcm(\overline{m},\overline{m'})$ for every pair of monomials in $G(I)$.

\begin{Lemma}
Let $I$ and $J$ be two copolar ideals. Then $\lcm(I)\cong \lcm(J)$.
\end{Lemma}

The $\lcm$-lattice of a monomial ideal encodes the structure of its minimal free resolution and thus its Betti numbers \cite{GPW99}. Some other important invariants are also fixed under polarization. One recent remarkable result in this direction is given in \cite{IKM17} where the authors proved that the Stanley conjecture can be reduced to the squarefree case via polarization, and that the Stanley projective dimension is invariant under polarization (in particular, two ideals with isomorphic $\lcm$-lattice have the same Stanley projective dimension).

The fact that several important properties like Betti numbers or the Cohen-Macaulay property are copolar is one of the main reasons to study polarizations of ideals.  A motivation for studying depolarization and ideals inside the same polarity class is to find some particular ideal in the class that can provide information about the rest of the ideals in the class.  

We first take advantage of the fact that the number of variables of the ambient ring is not constant within the same polarity class, but the projective dimension is.  Therefore, for any monomial ideal we can construct its depolarization poset and find the maximum elements whose ambient rings has the minimum number of variables. Since the number of variables of a polynomial ring is an upper bound for the projective dimension of its ideals, this procedure provides us with an upper bound for the projective dimension of the ideals in terms of their depolarization posets. Recall that the {\em width} of a poset is the maximum size of its antichains.

\begin{Theorem}\label{th:pdBound}
The  width of $\suppPos(I^P)$ is an upper bound for $\projdim(I)$.
\end{Theorem}
\begin{proof}
The projective dimension of $I$ is equal to the projective dimension of its polarization $I^P$ which is in turn the same as that of any of its depolarizations. Let $J$ be a depolarization of $I^P$ whose ambient ring has the smallest number of variables, say $r$. By Hilbert Syzygy Theorem we know that $\projdim(J)\leq r$. By Theorem \ref{th:allDepolarizations} we know that $r$ is given by the minimal number of paths in which we can partition the support poset of $I^P$ (observe that this number is the same for any $\suppPos_<(I)$ and $\suppPos(I)$). Since all paths are chains, by Dilworth's Theorem \cite{D50}, this number is smaller than the size of the maximal antichain of the support poset of $I^P$ which is the width of $\suppPos(I^P)$.
\end{proof}

An interesting question, although out of the scope of this paper, is to compare this bound with other bounds for the projective dimension of monomial ideals, like the ones in \cite{DHS13,DS15} and further explore the role of polarization with bounding invariants of monomial ideals.
\subsection{Ideals of nested type} To use depolarization, we usually study depolarization posets to find an ideal with a particularly nice property that can be transferred to its copolar ideals. For instance, here we study {\em ideals of nested type} \cite{BG06} to compute the algebraic invariants of their copolar ideals. 

A monomial ideal $I\subseteq \kb[x_1,\dots,x_n]$ is of nested type if each of its associated prime ideals is of form $\mathfrak{p}=(x_1,\dots,x_i)$ for some $i$. Ideals of nested type are also called quasi-stable \cite{S09b} or  ideals of {\em Borel type}  \cite{HPV03}. The equivalence between these families of ideals is not immediate and it has been proven by Seiler in \cite[Proposition~4.4]{S09b}.

The invariants of such ideals have been extensively studied in \cite{BG06,S09b}, and it is shown that their Castelnuovo-Mumford regularity and their projective dimension can be obtained in terms of their irreducible decompositions or in terms of their Pommaret bases. 
Moreover, a minimal free resolution of these ideals is explicitly computed in \cite[Theorem~8.6]{S09b}. 

Therefore, if a polarity class contains an ideal of nested type (i.e. quasi-stable), then we can use the aforementioned results to compute the Castelnuovo-Mumford regularity and the projective dimension of each ideal in such class. Since zero-dimensional ideals are of nested type we can make use of these considerations in the polarity classes that contain at least one zero-dimensional ideal. 

\smallskip

In the same spirit as Proposition~\ref{prop:supportIdeals} we provide a sufficient condition for a poset to be a support poset of  a zero-dimensional monomial ideal.
 
 \begin{Proposition}\label{prop:npaths}
Let $n, m_1,\dots, m_n$ be some positive integers with $1\leq m_i\leq n$ for all $i$ and let $m=\sum_{i}m_i$. Consider a poset $(\P,\subseteq)$ on subsets of $\{1,\dots,m\}$ formed by $n$ disjoint paths each of length $m_i$. If $n>2$ or $m_1=m_2$, then there is a squarefree monomial ideal $I$ whose support poset is $\P$. Moreover, if $m_i>1$ for all $i$, then there is a zero-dimensional monomial ideal copolar to $I$.
\end{Proposition}

\begin{proof}
Let $\P=A_1\sqcup\cdots \sqcup A_n$ where $A_i=\{\{a_{i,1}\},\{a_{i,1},a_{i,2}\},\dots,\{a_{i,1},\dots,a_{i,m_i}\}\}$. We assume that $n>2$ or $m_1\neq m_2$. The remaining case is studied in Example~\ref{exam:n=2}. 

For ease of notation, we identify each variable $x_{a_{i,j}}$ with its subindex $a_{i,j}$. We can assume without loss of generality that $m_1\geq \cdots \geq m_n$. 

We first construct a monomial ideal generated by the following sets of monomials:
\begin{enumerate}
\item $G_1$ consists of the monomials $\mu_i=a_{i,1}\cdots a_{i,{m_i}}$ for all $i$ with $m_i>1$.

\item $G_2$ consists of the monomials $\mu_{i,j}=a_{i,1}\cdots a_{i,j}b_{i,j}$  for all $i,j$ with $1<j<m_i$. Here, $b_{i,j}=a_{\lceil{i+j-1}\rceil,1}$ where $\lceil{i+j-1}\rceil$ denotes $i+j-1$ modulo $n$. Note that the indices $b_{i,j}$ are pairwise distinct for each $i$, as $m_i\leq n$.

\item $G_3=\{a_{i,1}a_{i',1}:\ a_{i,1},a_{i',1}\in G'_3\text{ and }a_{i,1}a_{i',1}\nmid m\ \text{for any}\ m\in G_1\cup G_2\}$, where $G'_3$ consists of all indices $a_{i,1}$ for $m_i=1$ that appeared at most once as $b_{j,k}$ in $G_2$, and indices $a_{i,1}$ for $m_i>1$ that never appeared as $b_{j,k}$ in $G_2$. 
\end{enumerate}
We now prove that the support poset of the ideal $I_G\subseteq \kb[x_{a_{1,1}},\dots,x_{a_{n,m_n}}]$ generated by the above sets of monomials is $(\P,\subseteq)$. 

To see this, first observe that, by construction, the monomials in $G=G_1\cup G_2\cup G_3$ do not divide each other. 
Now we show that for each pair of variables $a_{i,j}$, $a_{i',j'}$ with $i\neq i'$ there is at least one monomial in $G$ which only contains one of these variables. If $j>1$ or $j'>1$, then it is easy to find such a monomial in $G_1$. Now assume that $j=j'=1$. If $m_i>1$ or $m_{i'}>1$, then such a monomial can be found in $G_1$. Otherwise, they appear in separate monomials in $G_2$ or $G_3$.

Now, we show that $C_{a_{i,j}}=\{ a_{i,1},\dots,a_{i,j}\}$ for every variable $a_{ij}$. First note that every variable $a_{i,1}$ appears at least once in $G_2$ or $G_3$ without the rest of the variables $a_{i,j}$ for $j>1$. Thus $C_{a_{i,1}}=\{a_{i,1}\}$. For $m_i>1$, $a_{i,m_i}$ appears only in the monomial $\mu_i=a_{i,1}\cdots a_{i,m_i}$ in $G_1$, hence $C_{a_{i,m_i}}=\{ a_{i,1},\dots, a_{i,m_i}\}$. 
Now assume that $1<j<m_i$.  The variable $a_{i,j}$ appears always together with all the variables $a_{i,j'}$ for $j'<j$ since they all divide the monomials $\mu_i$ in $G_1$ and $\mu_{i,j}\in G_2$. On the other hand, if $m_i>\ell>j$, then by the construction of $G_2$, there is at last one monomial $\mu_{i,j}$ in which $a_{i,\ell}$ is not present. Hence, $C_{a_{i,j}}=\{ a_{i,1},\dots,a_{i,j}\}$ which completes the proof.

Moreover, if $m_i>1$ for each $i$, then using the chain partition given by the disjoint paths themselves, the corresponding depolarization of $I$ has one variable for each $i$ whose pure power appears in $G_1$, which implies that $I$ is zero-dimensional. 
\end{proof}

\begin{Example}\label{exam:n=2}
Let $n=2$, $m_1=2$ and $m_2=1$. Then the minimal generating set of any monomial ideal $I$ with the support poset  $\P=\{\{1\},\{1,2\},\{3\}\}$ must include a monomial divisible by $x_1x_2$. Therefore, the only candidates for such an ideal are $\langle x_1x_2x_3\rangle$, $\langle x_1x_2, x_3\rangle$, $\langle x_1x_2,x_1x_3\rangle$, $\langle x_1x_2,x_2x_3\rangle$ and $\langle x_1x_2, x_1x_3,x_2x_3\rangle$. However, none of them has $\P$ as it support poset.
\end{Example}
\begin{Example}
Consider the poset $(\P,\subseteq)$ on the following subsets of $\{1,\ldots,14\}$. Let $A_1=\{\{1\},\dots,\{1,2,3,4,5,6\}\},A_2=\{\{7\},\dots, \{7,8,9,10\}\},A_3=\{\{11\}\}, A_4=\{\{12\}\},A_5=\{\{13\}\},A_6=\{\{14\}\}$. Then following the notation of the proof of Proposition \ref{prop:npaths} we have that $(\P,\subseteq)$ is the support poset of the ideal generated by the monomials in $G_1\cup G_2 \cup G_3$, where
\begin{align*}
G_1=& \{x_1x_2x_3x_4x_5x_6, x_7x_8x_9x_{10}\}\\
G_2=&\{x_1x_2x_7,x_1x_2x_3x_{11},x_1x_2x_3x_4x_{12},x_1x_2x_3x_4x_5x_{13},x_7x_8x_{11},x_7x_8x_9x_{12}\}\\
G_3=&\{x_1x_{14}, x_{13}x_{14}\}\text{ and } G'_3=\{x_1,x_{13},x_{14}\}.
\end{align*}
\end{Example}


\begin{Example}\label{ex:quasistable}
Consider the following monomial ideal in 9 variables:
$$I=\langle x_1x_2x_3x_4,x_5x_6x_7,x_8x_9,x_1x_2x_3x_5,x_1x_2x_8,x_5x_6x_8,x_1x_5x_8\rangle.$$
By computing its associated primes we see that $I$ is not quasi-stable. However, by determining the depolarization poset of $I$ we found that the ideal, $$J=\langle y_1^4, y_2^3, y_3^2, y_1^3y_2,y_1^2y_3,y_2^2y_3,y_1y_2y_3\rangle\subseteq \kb[y_1,y_2,y_3]$$ is one of its maximum depolarizations. As $J$ is a zero-dimensional ideal, and hence quasi-stable, by applying the results of \cite{S09b} we obtain $\pd(J)=2$ and $\reg(J)=5$, thus obtaining these invariants for $I$.
\end{Example}

\subsection{Polarization and free resolutions}
We finish this section with some considerations on polarization and free resolutions. The results in this subsection will be used in Section \ref{sec:systems} when we study the algebraic analysis of system reliability. First we define the polarization of a resolution. Let 
\[\FF: \cdots \xrightarrow{\ } F_i\xrightarrow{\delta_i} F_{i-1}\xrightarrow{\delta_{i-1}}F_{i-2} \xrightarrow{\ } \dots\]
be a multigraded chain complex of $R$-modules, i.e., $F_i=\bigoplus_{j=1}^{r_i}R(-\mu_{i,j})$ where $\mu_{ij}\in \NN^n$ and the differentials $\delta_i$ have multidegree $0$. The differentials $\delta_{i}$ are given by matrices $A_i$ whose entries are monomials in $\NN^n$. We denote by $e_{i,j}$ the standard generator of the $j$-th summand of $F_{i}$ whose multidegree is $\mu_{i,j}$. Then the $j$-th column of $A_i$ is given by $(a^i_{1,j},a^i_{2,j},\dots,a^i_{r_{i-1},j})$ where $\delta(e_{i,j})=\sum_{k=1}^{r_{i-1}}a^i_{k,j}e_{i-1,k}$ and the $e_{i-1,k}$ are the standard generators of $F_{i-1}$. The nonzero entries $a^i_{k,j}$ are given by $\mu_{i,j}/\mu_{i-1,k}$.

\begin{Definition}
We define $\overline{\FF}$, the {\em polarization} of $\FF$, as the chain complex given by
\[\overline{\FF}: \cdots \xrightarrow{\ } \overline{F}_i\xrightarrow{\overline{\delta}_i} \overline{F}_{i-1}\xrightarrow{\overline{\delta}_{i-1}}\overline{F}_{i-2} \xrightarrow{\ } \dots\]
where $\overline{F}_i=\bigoplus_{j=1}^{r_i}R(-\overline{\mu}_{ij})$ if $F_i=\bigoplus_{j=1}^{r_i}R(-\mu_{ij})$ and the matrices $\overline{A}_i$ of the differentials $\overline{\delta}_i$ are given by
 \[
 \overline{a}^i_{j,k}=
 	\begin{cases}
      		0 &\quad\text{if}\  a^i_{j,k}=0\\
		x_{1,{c_1+1}}\cdots x_{1,b_1}\cdots x_{n,{c_n+1}}\cdots x_{n,b_n}& \quad\text{if } a^i_{j,k}\neq 0,\, \mu_{i,j}=\xb^b,\ \mu_{i-1,k}=\xb^c,
            \end{cases}
 \]
 where $b=(b_1,\ldots,b_n)$ and $c=(c_1,\ldots,c_n)$. Note that if $0\neq a^i_{j,k}\in\kb$ then $\overline{a}^i_{j,k}=a^i_{j,k}$.
 \end{Definition}
By polarizing a resolution of a monomial ideal we obtain a resolution of its polarization. This is a consequence of Theorem 3.3 in \cite{GPW99}. Here, we present our own proof to keep track of the explicit changes in each multidegree. The use of depolarization to compute resolutions of a monomial ideal form the resolutions of its polarization is also a well-known result, see Examples 3.4 in \cite{GPW99} and see \cite{F82}.

\begin{Proposition}\label{prop:resolution}
Let $\FF$ be a multigraded free resolution of a monomial ideal $I$. The polarization $\overline{\FF}$  of $\FF$ is a multigraded free resolution of the polarization of $I$. Moreover, the ranks and the graded ranks of $\overline{\FF}$ are equal to those of $\FF$. In the case of multigraded ranks, we have that the $(i,\mu)$-rank of $\FF$ equals the $(i,\overline{\mu})$-rank of $\overline{\FF}$.
\end{Proposition}
\begin{proof}
Given the fact that $\FF$ is a multigraded free resolution of $I$ and the construction of $\overline{\FF}$ by polarization, we only need to prove that ${\rm Im}(\overline{\delta}_i)={\rm Ker}(\overline{\delta}_{i-1})$.

We know that $\delta^2=0$. Explicitly, we have that
\[\delta_{i-1}\delta_i(e_{i,j})=\sum_{k=1}^{r_{i-1}} \sum_{l=1}^{r_{i-2}} a^{i}_{k,j} a^{i-1}_{l,k} e_{i-2,l}=0 \;\forall\, i,j. \]
This implies that
\begin{equation}
\label{eq:coef_1}
\sum_{k=1}^{r_{i-1}} \sum_{l=1}^{r_{i-2}} a^{i}_{k,j} a^{i-1}_{l,k}=\sum_{k=1}^{r_{i-1}} \sum_{l=1}^{r_{i-2}} \frac{\mu(e_{i,j})}{\mu(e_{i-1,k})}\frac{\mu(e_{i-1,k})}{\mu(e_{i-2,l})}=0 \;\forall\, i,j.
\end{equation}
On the other hand, from the definitions of the maps in $\overline{\FF}$ we have that for any $i,j$
\[\overline{\delta}_{i-1}\overline{\delta}_i(\overline{e}_{i,j})=\sum_{k=1}^{r_{i-1}} \sum_{l=1}^{r_{i-2}} \overline{a}^i_{k,j}\overline{a}^{i-1}_{l,k}\overline{e}_{i-2,l}.\]
Now by polarizing (\ref{eq:coef_1}), we obtain
\[\sum_{k=1}^{r_{i-1}} \sum_{l=1}^{r_{i-2}} \frac{\overline{\mu}(e_{i,j})}{\overline{\mu}(e_{i-1,k})}\frac{\overline{\mu}(e_{i-1,k})}{\overline{\mu}(e_{i-2,l})}=\sum_{k=1}^{r_{i-1}} \sum_{l=1}^{r_{i-2}} \overline{a}^i_{k,j}\overline{a}^{i-1}_{l,k}=0,\]
and hence $\overline{\delta}_{i-1}\overline{\delta}_i(\overline{e}_{i,j})=0$ for all $i,j$.
Since polarization induces a multigraded isomorphism, the result follows.
\end{proof}

Proposition \ref{prop:resolution} is important in our context since it allows us to use polarization in the algebraic analysis of system reliability and obtain formulas and bounds for the reliability of the system corresponding to the polarization of a given ideal. In particular, we can use the so-called Mayer-Vietoris trees as one of the main tools applied in \cite{SW09,SW10}. Mayer-Vietoris trees are a way to encode the support, i.e., the collection of multidegrees of the generators of the free modules of the so-called mapping cone resolution \cite {CE95, DM14} of a monomial ideal. We refer the reader to \cite{S09} for the definition and the basic properties of the Mayer-Vietoris trees.

\begin{Corollary}
Let $\FF$ be a cone resolution of a monomial ideal $I$, then $\overline{\FF}$ is a cone resolution of the polarization of $I$. In particular, if $\TT$ is a Mayer-Vietoris tree of $I$, then the polarized tree $\overline{\TT}$ is a Mayer-Vietoris tree of the polarization of $I$.
\end{Corollary}

\section{Multi-state systems via binary systems and vice versa}\label{sec:systems}
We now turn to the application of monomial ideals to multi-state system analysis, which is the motivation of this work. The algebraic approach to system reliability, developed by the authors for the binary systems, gives an insight on the structure of the systems under study besides providing good computational tools to obtain reliability polynomials and bounds. In the next few examples, we show how the structure of multi-state systems can be analyzed by algebraic means and that this analysis can be transferred between binary and multi-state systems using polarization and depolarization.

\subsection{Multi-state coherent systems}\label{subset:systems}
In reliability theory \cite{BP75,BK94,KZ03,N11}, a {\em system} $S$ is given by a set of components, say $n$, denoted by $c_i$ for $i\in\{1,\dots,n\}$. Each $c_i$ can be in a discrete number of ordered states, i.e. levels of performance, $\Sc_i=(0,\dots,m_i)$. The system itself has $m+1$ possible states $\Sc=(0,\dots,m)$ for some integer $m$. The states of the system measure the overall performance of the system. In this paper, we assume that the system (respectively the component) in state $j$ represents {\it better performance} than the system (respectively the component) in state $i$, whenever $j>i$. We define a structure function $\phi:\Sc_1\times \cdots \times \Sc_n\rightarrow \Sc$ that for each $n$-tuple of component states, outputs a state of the system. We say that the system is {\it coherent} if $\phi(\xb)\geq\phi(\yb)$ whenever $\xb>\yb$, which means $x_i\geq y_i$ for every $i$ and there is at least one index $i$ such that $x_i>y_i$. Conversely, $\phi(\xb)\leq\phi(\yb)$ whenever $\xb<\yb$. Examples of coherent systems include electrical and transport networks, pipelines, biological and industrial systems among many others \cite{KZ03}.  If $m_i=1$, then we say that component $c_i$ is {\it binary}. If $m=1$, then we say that the system is binary. We have therefore the following types of systems with respect to their number of states:
\begin{itemize}
\item[-] If $m=1$ and $m_i=1$ for all $i$, we have a binary system with binary components. These are usually simply referred to as binary systems.
\item[-] If $m>1$ and $m_i=1$ for all $i$, we have a multi-state system with binary components.
\item[-] If $m=1$ and there is at least one $i$ with $m_i>1$, we have a binary system with multi-state components.
\item[-] If $m>1$ and there is at least one $i$ with $m_i>1$, we have a multi-state system with multi-state components.
\end{itemize}
We follow the notation in \cite{GN17,N11}. However, we allow a more general kind of system by not restricting to the case  $\max(\Sc)\leq\max(\Sc_i)$ for all  $i$. For other definitions of multi-state system and a review of multi-state reliability analysis, see \cite{LL03,YJ12} and the references therein.

\subsection{The algebraic method in reliability analysis}\label{subset:algebraicreliability}
Let $S$ be a coherent system with $n$ components. Let $0<j\leq m$, we denote by $\Fc_{S,j}$ the set of tuples of components' states $\xb$ such that $\phi(\xb)\geq j$. The elements of $\Fc_{S,j}$ are called {\it $j$-working states} or {\it $j$-paths} of $S$. If $m=1$, then we simply speak of {\it working states} or {\it paths}. The tuples of components' states  $\xb$ with $\phi(\xb)<j$ are called {\it $j$-failure states} or {\it $j$-cuts}, respectively {\it failure states} or {\it cuts} for $m=1$. Let $\overline{\Fc}_{S,j}$ be the set of minimal $j$-working states (minimal $j$-paths), i.e., states in $\Fc_{S,j}$ such that degradation of the performance of any component provokes that the overall performance of the system is degraded to $j'<j$.

Now, let $R=\kb[x_1,\dots,x_n]$ be a polynomial ring over a field $\kb$. Each tuple of components' states $(s_1,\dots,s_n)\in \Sc_1\times\cdots \times \Sc_n$ corresponds to the monomial  $x_1^{s_1}\cdots x_n^{s_n}$ in $R$. The {\it coherent property} of the system is equivalent to saying that the elements of $\Fc_{S,j}$ correspond to the monomials in an ideal, denoted by $I_{S,j}$ and called the {\it $j$-reliability ideal} of $S$. The unique minimal monomial generating set of $I_{S,j}$ is formed by the monomials corresponding to the elements of $\overline{\Fc}_{S,j}$ (see \cite[\S 2]{SW09} for more details). Hence, obtaining the set of minimal $j$-paths of $S$ is equivalent to compute the minimal generating set of $I_{S,j}$.  

In order to compute the {\it $j$-reliability} of $S$, i.e., the probability that the system is performing at least at level $j$, we can use the numerator of the Hilbert series of $I_{S,j}$, denoted by $H_{I_{S,j}}$. The polynomial $H_{I_{S,j}}$ gives a formula, in terms of $x_1,\dots,x_n$ that enumerates all the monomials in $I_{S,j}$, i.e., the monomials corresponding to the states in $\Fc_{S,j}$. Hence, computing the (numerator of) the Hilbert series of $I_{S,j}$ provides a method to compute the $j$-reliability of $S$ by substituting $x_i^a$ by $p_{i,a}$, the probability that the component $i$ is at least performing at level $a$, as explored in \cite[\S 2]{SW09} (for the binary case).

Often in practice it is more useful to have {\it bounds} on the $j$-reliability of $S$ rather than the complete precise formula. In order to have a formula that can be truncated at different summands to obtain bounds for the $j$-reliability in the same way that we truncate the inclusion-exclusion formula to obtain the so-called Bonferroni bounds, we need a special way to write the numerator of the Hilbert series of $I_{S,j}$. This convenient form is given by the alternating sum of the ranks in any free resolution of the ideal $I_{S,j}$. Every monomial ideal has a {\em minimal} free resolution, which provides the tightest bounds among the aforementioned ones. In general, the closer the resolution is to the minimal one, the tighter the bounds obtained, for full details see, e.g., \cite[\S 3]{SW09}.

In summary, the algebraic method for computing the $j$-reliability of a coherent system $S$ works as follows: 
\begin{enumerate}
\item{Associate to the system $S$ its $j$-reliability ideal $I_{S,j}$}.
\item{Obtain the minimal generating set of $I_{S,j}$ to get the set $\overline{\Fc}_{S,j}$.}
\item{Compute the Hilbert series of $I_{S,j}$ to have the $j$-reliability of $S$.}
\item[(3')]{Compute any free resolution of $I_{S,j}$. The alternating sum of the ranks of this resolution gives a formula for the Hilbert series of $I_{S,j}$ i.e., the unreliability of $S$, which provides bounds by truncation at each summand.}
\end{enumerate}

The choice between steps (3) or (3') depends on our needs. If we are only interested in computing the full reliability formula, then we can use any algorithm that computes Hilbert series in step (3). However, if we need bounds for our system reliability, then 
we can compute any free resolution of $I_{S,j}$ and thus perform step (3'). If the performing probabilities of different components are independent and identically distributed (i.i.d), then in points (3) and (3') of this procedure we only need the graded version of  Hilbert series and free resolutions. Otherwise, we need their multigraded version. For more details and the proofs of the results described here, we refer to \cite{SW09,SW12}. To see more applications of this method in reliability analysis we refer to \cite{SW10,SW11,SW15}.

We can study multi-state systems via binary systems and vice versa by means of polarization and depolarization of their $j$-reliability ideals. The main reason behind this approach is that 
the Hilbert series and free resolutions of monomial ideals and their polarizations are related, see Proposition \ref{prop:copolarProperties}.
For a complete application of the polarization and depolarization operations in the algebraic method, we also need the statement that the ranks of the modules in any resolution of a monomial ideal and its polarization are the same, see Proposition \ref{prop:resolution}. When using the polarization of a $j$-reliability ideal to study the system's reliability, we have to carefully adapt the probability associated to the monomials in the new ideal. Under independence assumption, the term $x_1^{a_1}x_2^{a_2}$ contributes $\prob(c_1\geq a_1)\cdot \prob(c_2\geq a_2)$ to the reliability of the system. If independence is not assumed, then we need to individually study the probability evaluation of each monomial. In general, one needs to know the full distribution on the failure set, although the structure of the sets are distribution-free. In the case of polarization of a system reliability ideal, we have to take care of monomials that include products of the type $x_{i,1}\cdots x_{i,k}$ which must be evaluated as $\prob(c_i\geq k)$.

\subsection{Examples}\label{subsec:examples}
Our first two examples apply the algebraic method to the analysis of the reliability of multi-state coherent systems. We show that our approach, using the algebraic method, can be used to analyze the reliability of such systems in an efficient and clear way. The third and fourth examples we propose, will demonstrate how  the depolarization method can be used to compute the reliability of various systems. In particular,  in the fourth example we demonstrate how depolarizing a family of system ideals would make the reliability computations faster.

\medskip

\noindent {\bf Decreasing MS $k$-out-of-$n$ system:} \label{exam:1}
This example is taken from \cite{HZW00} in which the authors define generalized multi-state $k$-out-of-$n$ systems (denoted MS $k$-out-of-$n$) as $n$-component systems with $\phi(\xb)\geq j$ ($j=0,\dots,m)$, if there exists an integer value $l$ ($j\leq l\leq m$) such that at least $k_l$ components are in states at least as good as $j$. In that paper the authors describe ad-hoc methods for computing the reliability of MS $k$-out-of-$n$ systems. Different computations are proposed for the cases that the system is increasing or decreasing (which means that the sequence of $k_l$ is respectively increasing or decreasing) and also different computations need to be done if the components are identically distributed or not. For instance, Example 8 in \cite{HZW00} is an MS $k$-out-of-$3$ system with four states $(0,1,2,3)$ such that $k_3=2$, $k_2=2$ and $k_1=3$, i.e., the system is
\begin{itemize}
\item[-] In state $3$ or above if at least $2$ components are in state $3$ or above.
\item[-] In state $2$ or above if at least $2$ components are in state $2$ or above.
\item[-] In state $1$ or above if all $3$ components are in state $1$ or above, or if at least $2$ components are in state $2$ or above, or if at least $2$ components are in state $3$ or above.
\end{itemize}
This is called a {\em decreasing} MS $k$-out-of-$n$ system, because the sequence of $k_l$ is a decreasing one. In this example the probabilities of different components are: $p_{1,0}=0.1$, $p_{1,1}=0.2$, $p_{1,2}=0.3$, $p_{1,3}=0.4$, $p_{2,0}=0.1$, $p_{2,1}=0.2$, $p_{2,2}=0.2$, $p_{2,3}=0.6$, $p_{3,0}=0.1$, $p_{3,1}=0.2$, $p_{3,2}=0.4$, $p_{3,3}=0.3$.
Here, we use the algebraic method to compute the probability that the system is respectively in states $0$, $1$, $2$ and $3$, we and obtain the same exact results as in \cite{HZW00}.
\begin{itemize}
\item[-] For the system to be in state $3$, there must be at least  $2$ components in state $3$ or above ($k_3=2$). Hence, the corresponding ideal is $I_{S_3}=\langle x^3y^3,x^3z^3,y^3z^3\rangle$. The numerator of the Hilbert series is $H_{I_{S_3}}=x^3y^3+x^3z^3+y^3z^3-2(x^3y^3z^3)$ and when substituting the probabilities, we have that the probability that the system is in state $3$ or above, denoted by $R_{S,3}$, is $0.396$, which equals the probability that the system is exactly in state $3$, denoted $r_{S,3}$.
\item[-] The system is in state $2$ or above if at least $2$ components are in state $2$ or above, hence $I_{S_2}=\langle x^2y^2,x^2z^2,y^2z^2\rangle$. The numerator of the Hilbert series is $H_{I_{S_2}}=x^2y^2+x^2z^2+y^2z^2-2(x^2y^2z^2)$ and we obtain $R_{S,2}=0.826$ and $r_{S,2}=R_{S,2}-R_{S,3}=0.826-0.396=0.430$.
\item [-] Since $k_1=3$, the system is in state $1$ or above if all $3$ components are in state $1$ or above or if at least $2$ components are in state $2$ or above, or if at least $2$ components are in state $3$ or above. The corresponding ideal is then $I_{S_1}=\langle xyz,x^2y^2,x^2z^2,y^2z^2\rangle$ with $H_{I_{S_1}}=xyz+x^2y^2+x^2z^2+y^2z^2-(xy^2z^2+x^2yz^2+x^2y^2z)$. Thus, we obtain $R_{S,1}=0.89$ and $r_{S,1}=R_{S,1}-R_{S,2}=0.89-0.826=0.064$.
\item[-] Finally $r_{S,0}=R_{S,0}-R_{S,1}=1-0.89=0.11$.
\end{itemize}
With respect to the computational applicability of the algebraic method, the authors in \cite[page 109]{HZW00} indicate that ``For most practical engineering problems, a limited state number $M$, for example $M=10$, is big enough to describe the performances of the system and its components''. We wrote a program in the computer algebra system \texttt{Macaulay2} \cite{M2} to compute the Hilbert series of all the $j$-reliability ideals in a decreasing MS $k$-out-of-$n$ system. The program computes the reliability polynomials for these systems with $10$ components and $10$ levels of performance in less than a minute. Figure \ref{fig:reliabilities} shows the reliability polynomials of a $k$-out-of-$10$ decreasing MS system for which $k_1$ to $k_{10}$ are $(9,8,7,6,5,5,4,4,3,2)$. The program took $40$ seconds on a laptop\footnote{CPU: intel i7-4810MQ, 2.80 GHz. RAM: 16Gb}, hence the method is practical.

\begin{figure}[t]
\begin{center}
\includegraphics[scale=0.65]{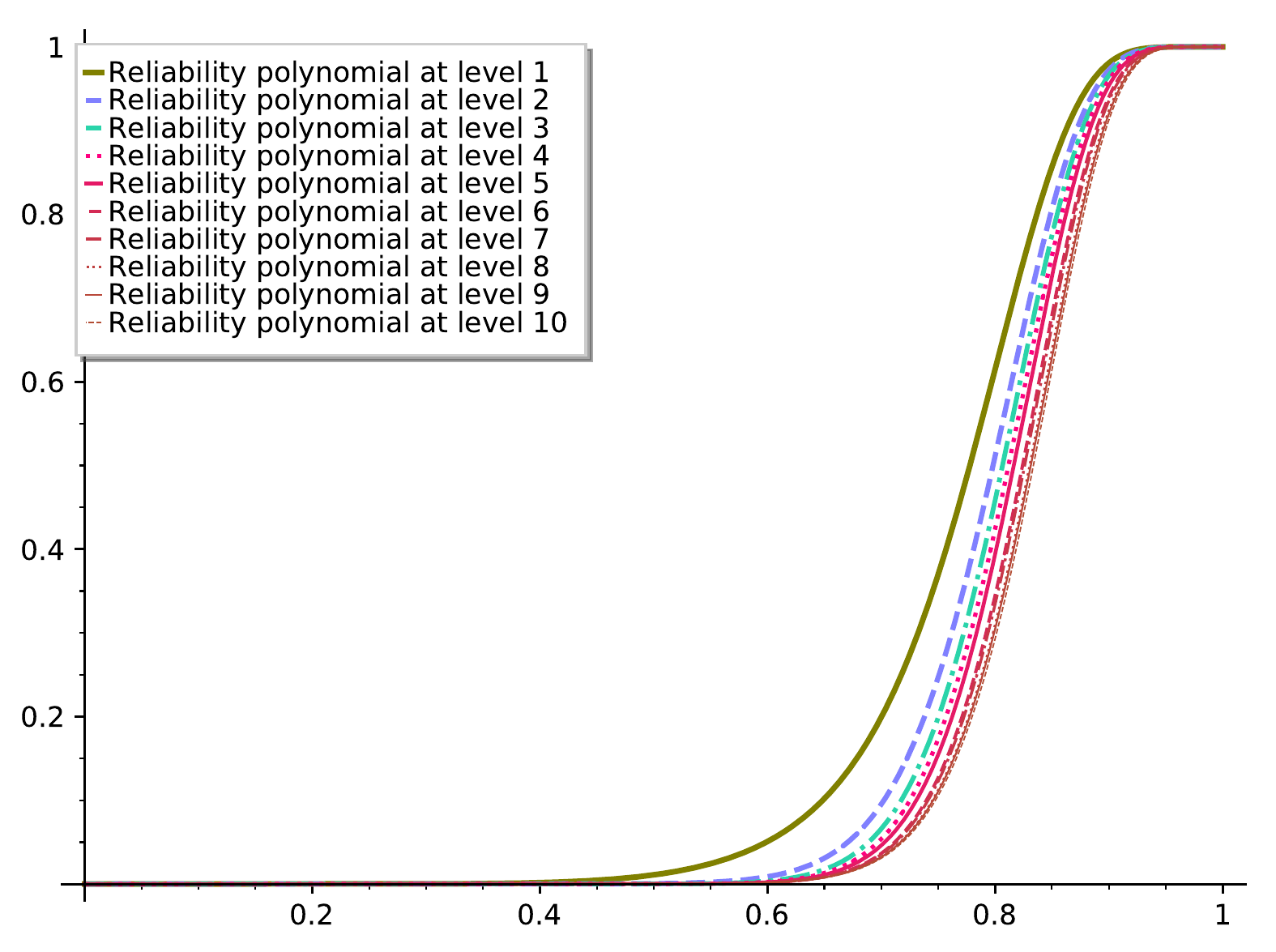}
\caption{Reliability polynomials at all levels for a decreasing MS $k$-out-of-$10$ system with $(k_1,\dots,k_{10})=(9,8,7,6,5,5,4,4,3,2)$.}\label{fig:reliabilities}
\end{center}
\end{figure} 

\smallskip

\noindent {\bf Flow network: }\label{exam:flowNetwork}
A flow nework $S$ has $n$ components $c_i$, $i=1,\dots,n$, where each of them can be in $m_i+1$ states $\Sc_i=\{0,1,\dots,m_i\}$ and the set of states of the system is $\Sc=\{0,\dots,m_1+\cdots+m_n\}$. The structure function of $S$ is $\phi=x_1+\cdots+x_n$. The $j$-reliability ideal $I_j\subseteq R=\kb[x_1,\dots,x_n]$ of $S$ is generated by all monomials in $R$ of degree $j$. These ideals are stable and therefore, the resolution given in \cite{EK90} is minimal and provides a formula for the $j$-reliability of $S$ which can be truncated to obtain bounds.

The authors of \cite{GN17} point out that in more complex systems, the computation of reliability bounds, denoted by $l'^{j}_\phi(\pb)$, $l^{**j}_\phi(\pb)$ and $\tilde{l}^j_\phi(\pb)$, can become quite difficult and computationally expensive. A wide class of those more complicated systems are those in which we can by some means obtain the minimal cuts or paths but their structure is complicated. In these cases the algebraic approach can be a very useful tool.

We consider now an example of a flow network with different levels of performance, see Example 2 in \cite{GN17}. The system $S$ has two components and each of these can be in three states, $\Sc_1=\Sc_2=\{0,1,2\}$. The system itself can be in five states, $\Sc=\{0,1,2,3,4\}$. The structure function of this system is $\phi(\xb)=x_1+x_2$, i.e., the state of the system is the sum of the states of each of its components. The probability that each component is at least in state $1$ is $p^{(1)}=0.9$ and the probability that each component is in state $2$ is $p^{(2)}=0.8$. In \cite{GN17} the authors give several bounds for the reliability of the system at level $j$. One of these bounds is based on minimal paths (denoted $l'^{j}_\phi(\pb)$), another one is based on minimal cuts (denoted $l^{**j}_\phi(\pb)$) and the third one is proposed by the authors and denoted $\tilde{l}^j_\phi(\pb)$. In this case, since the components' probabilities are i.i.d., we know that $\tilde{l}^j_\phi(\pb)$ is sharp \cite{GN17}.

We use now the algebraic method to compute the $j$-reliabilities of this system:
\begin{itemize}
\item[-]{For performance at level $1$ the minimal paths are $(1,0)$ and $(0,1)$, the corresponding ideal is $I_1=\langle x, y\rangle$ whose Hilbert function is $H_{I_1}=x+y-xy$. By substituting the corresponding probabilities we have that $R_{S,1}=0.9+0.9-0.9\cdot 0.9=0.99$}.
\item[-]{For performance at level $2$ the minimal paths are $(2,0)$, $(1,1)$ and $(0,2)$, the corresponding ideal is $I_2=\langle x^2,xy, y^2\rangle$ whose Hilbert function is $H_{I_2}=x^2+xy+y^2-(x^2y+xy^2)$. By substituting the corresponding probabilities we have that $R_{S,2}=0.8+0.9\cdot 0.9+0.8-(0.8\cdot 0.9+0.9\cdot0.8)=0.97$}.
\item[-]{For performance at level $3$ the minimal paths are $(1,2)$ and $(2,1)$, the corresponding ideal is $I_2=\langle x^2y,xy^2\rangle$ whose Hilbert function is $H_{I_2}=x^2y+xy^2-x^2y^2$. By substituting the corresponding probabilities we have that $R_{S,3}=0.8\cdot 0.9+0.9\cdot0.8-0.8\cdot 0.8=0.80$}.
\item[-]{Finally, for performance at level $4$ the only minimal path is $(2,2)$, $I_4=\langle x^2y^2\rangle$, $H_{I_4}=x^2y^2$ and we have that $R_{S,4}=0.8\cdot 0.8=0.64$}.
\end{itemize}


\smallskip

\noindent{\bf Coherent system given by structure function:}\label{exam:2}
Let $S$ be the coherent system with $4$ components $c_1,c_2,c_3,c_4$ such that $c_1,c_3,c_4$ have two possible states $0$ and $1$ meaning failure and working, while $c_2$ has three possible states $0,1,2$. The system $S$ itself can be in two possible states, working $(1)$ or failure $(0)$. The probabilities $p_{i,j}$ that component $i$ is in state $j$ are: $p_{1,0}=0.2$, $p_{1,1}=0.8$, $p_{2,0}=0.3$, $p_{2,1}=0.2$, $p_{2,2}=0.5$, $p_{3,0}=0.1$, $p_{3,1}=0.9$, $p_{4,0}=0.1$, $p_{4,1}=0.9$. 
\begin{table}[h]
\begin{tabular}{r|cccccccccccccccccccccccc}
\hline
$c_1$&$0$&$1$&$0$&$0$&$0$&$0$&$1$&$1$&$1$&$1$&$0$&$0$&$0$&$0$&$0$&$1$&$1$&$1$&$1$&$1$&$0$&$1$&$1$&$1$\\
$c_2$&$0$&$0$&$1$&$2$&$0$&$0$&$1$&$2$&$0$&$0$&$1$&$1$&$2$&$2$&$0$&$1$&$2$&$1$&$2$&$0$&$1$&$2$&$1$&$2$\\
$c_3$&$0$&$0$&$0$&$0$&$1$&$0$&$0$&$0$&$1$&$0$&$1$&$0$&$1$&$0$&$1$&$1$&$1$&$0$&$0$&$1$&$1$&$1$&$1$&$1$\\
$c_4$&$0$&$0$&$0$&$0$&$0$&$1$&$0$&$0$&$0$&$1$&$0$&$1$&$0$&$1$&$1$&$0$&$0$&$1$&$1$&$1$&$1$&$1$&$1$&$1$\\
\hline
$\phi(\xb)$&$0$&$0$&$0$&$1$&$0$&$0$&$1$&$1$&$1$&$0$&$1$&$0$&$1$&$1$&$1$&$1$&$1$&$1$&$1$&$1$&$1$&$1$&$1$&$1$\\
\hline
\end{tabular}
\medskip
\caption{Structure function $\phi(\xb)$ for system $S$.}
\label{table:example1}
\end{table}

We want to study the reliability of $S$ whose structure function $\phi$ is given in Table \ref{table:example1}. One can see from the table that
\[
\overline{\Fc_{S}}=\{(1,1,0,0),(1,0,1,0),(0,2,0,0),(0,1,1,0),(0,0,1,1)\}.
\]
Hence, the reliability ideal of $S$ is $I_{S}=\langle xy,xz,y^2,yz,zt\rangle$. The numerator of the Hilbert series of $I_S$ given by the alternating sum of its Betti numbers is 
\[
H_{I_S}=xy+xz+y^2+yz+zt-(2xyz+xy^2+xzt+y^2z+yzt)+xy^2z+xyzt.
\]
Substituting the probabilities in $H_{I_S}$, we obtain that $R_S=0.9606$. 
On the other hand, the polarization of $I_{S}$ is
\[
I^P_{S}=\langle x_1y_1,x_1z_1,y_1y_2,y_1z_1,z_1t_1 \rangle\subset \kb[x_1,y_1,y_2,z_1,t_1].
\]
Since $I^P_{S}$ is squarefree, we can use its Stanley-Reisner complex $\Delta_{I^P_S}$ to study its algebro-combinatorial features. The facets of the corresponding simplicial complex $\Delta_{I^P_S}$ are $\{x_1,y_2,t_1\},\{y_2,z_1\}$ and $\{y_1,t_1\}$. We can see that the numerator of the Hilbert series of $I^P_S$ is
\begin{eqnarray*}
H_{I_S^P}&=&x_1y_1+x_1z_1+y_1y_2+y_1z_1+z_1t_1\\ & &-(2x_1y_1z_1+x_1y_1y_2+x_1z_1t_1+y_1y_2z_1+y_1z_1t_1)\\
& &+x_1y_1y_2z_1+x_1y_1z_1t_1.
\end{eqnarray*}

Now, we substitute the probabilities, taking into account that $y_1y_2$ corresponds to $\prob(c_2\geq2)$. We obtain $R_S=0.9606$.

\smallskip

Studying the depolarization operation on $I^P_S$ we find that we can use the following sets for depolarizing ideal $I^P_S$  
\[
\sigma_{x_1}=\{x_1\},\ \sigma_{y_1}=\{y_1\},\ \sigma_{y_2}=\{y_1,y_2\},\ \sigma_{z_1}=\{z_1\}\ \text{and}\ \sigma_{t_1}=\{z_1,t_1\}.
\]
Hence, using the partition $\{x_1\},\{y_1,y_2\},$ $\{z_1,t_1\}$ we obtain a depolarization of $I^P_S$  in only three indeterminates, as  $J=\langle ab,ac,b^2,bc,c^2\rangle \subset \kb[a,b,c]$.  The numerator of the Hilbert series of this ideal is 
\[
H_J=ab+ac+b^2+bc+c^2-(2abc+ab^2+ac^2+b^2c+bc^2)+ab^2c+abc^2.
\]
In order to use this expression to evaluate the reliability of $S$ we must keep track of the meaning of the new variables in terms of the ones in $I^P_S$, i.e., the monomial $b^2$ corresponds to $y_1y_2$ which corresponds to $\prob(c_2\geq 2)$ but $c^2$ corresponds to $z_1t_1$ which is evaluated as $\prob(c_3\geq1)\cdot\prob(c_4\geq1)$. Using these evaluations we obtain the same result that $R_S=0.9606$.

\smallskip

\noindent{\bf Depolarization of consecutive $k$-out-of-$n$ systems:}\label{exam:3}
A consecutive $k$-out-of-$n$:G system \cite{KZ03} is a binary system with $n$ components that works whenever $k$ {\em consecutive} components work. The reliability ideal of such a system is $J_{k,n}=\langle x_1\cdots x_k,\dots,x_{n-k+1}\cdots x_n\rangle\subseteq\kb[x_1,\dots,x_n]$. The ideal $J_{k,n}$ has $n-k+1$ generators, all of degree $k$ in $n$ variables, which are corresponding to the set of all $k$-paths of the line graph \cite{HV10}. The depolarization poset of $J_{k,n}$ has a maximal element $J'_{k,n}$ whose ambient ring has $n+2-2k$ variables, and we can use it to compute the reliability of consecutive $k$-out-of-$n$:G systems when $n$ is large.

Table \ref{table:consecutive} shows the timings of an algorithm implemented by the authors using the Hilbert series implementation in \cite{M2} to compute the reliability polynomial of several large consecutive $k$-out-of-$n$ systems. The third column in the table shows the times used to compute the  reliability polynomial using the original squarfree ideal  $J_{k,n}$ and the fourth column shows the times used to compute the reliability polynomial using the maximal depolarization $J'_{k,n}$. The times are in seconds. OOT means the computation was manually stopped after $24$ hours. Observe that the times are reduced due to the reduction of the number of variables in the ambient ring. Working with the maximal depolarization makes it possible to handle bigger cases that are not possible to deal with using the squarefree reliability ideals.
\begin{table}[h]
\begin{tabular}{cccrr}
$n$&$k$&Num. gens.&Time $J_{k,n}$&Time $J'_{k,n}$\cr
\hline
$100$&$30$&$71$&$0.54$&$0.18$\cr
$100$&$15$&$86$&$34.62$&$17.81$\cr
$200$&$60$&$141$&$8.12$&$1.63$\cr
$200$&$30$&$171$&$1936.16$&$883.81$\cr
$300$&$90$&$211$&$56.12$&$8.63$\cr
$300$&$45$&$256$&OOT&$11941.60$\cr
\hline
\end{tabular}
\medskip
\caption{Computing times for the reliability polynomials of several consecutive $k$-out-of-$n$:G system ideals and their maximal depolarizations.}
\label{table:consecutive}
\end{table}

\section*{Acknowledgements}
The authors are partially funded by grant MTM2017-88804-P of Ministerio de Econom\'ia, Industria y Competitividad (Spain). The first author was partially supported by EPSRC grant EP/R023379/1. We are grateful to the anonymous referees for very helpful comments on the earlier version of this paper.

{}
\end{document}